\documentclass[11pt,letterpaper]{article}
\usepackage{amsfonts, amsmath, amssymb, amscd, amsthm, graphicx}

\makeatletter
\def\blfootnote{\xdef\@thefnmark{}\@footnotetext}
\makeatother

\newtheorem{thm}{Theorem}[section]

\newtheorem{lem}[thm]{Lemma}
\newtheorem{prop}[thm]{Proposition}

\theoremstyle{definition}

\theoremstyle{remark}

\newfont{\eufm}{eufm10}

\newcommand{\e }{\varepsilon }

\renewcommand{\phi }{\varphi}
\renewcommand{\kappa }{\varkappa}
\newcommand{\Hl }{\{ H_\lambda \} _{ \lambda \in \Lambda  }}

\renewcommand{\t }{{\rm Th}_{\forall} }

\begin{document}

\title{On the universal theory of torsion and lacunary hyperbolic groups}
\author{D. Osin \thanks{This work has been supported by the NSF grant DMS-0605093. }}
\date{}
\maketitle

\begin{abstract}
We show that the universal theory of torsion groups is strongly contained in the universal theory of finite groups. This answers a question of Dyson. We also prove that the universal theory of some natural classes of torsion groups is undecidable. Finally we observe that the universal theory of the class of hyperbolic groups is undecidable and use this observation to construct a lacunary hyperbolic group with undecidable universal theory. Surprisingly, torsion groups play an important role in the proof of the latter results.
\bigskip


\medskip

\end{abstract}

\section{Introduction}

Recall that a {\it universal sentence } in a first order language $\mathcal L$ is any sentence of the form
$$
\forall x_1\ldots \forall x_k \, \Phi,
$$
where $\Phi $ is a quantifier free formula. The {\it universal theory} of a class of groups $\mathcal C$ is defined to be the set of all universal sentences in the first order group theoretic language that hold in all groups from $\mathcal C$. As usual, we denote the universal theory of $\mathcal C$ by by $\t (\mathcal C)$.

Clearly if $\mathcal C$ contains another class, say $\mathcal D$, then $\t (\mathcal C)\subseteq \t (\mathcal D)$. In particular,
\begin{equation}\label{tf}
\t(\mathcal T)\subseteq\t(\mathcal F),
\end{equation}
where $\mathcal T$ and $\mathcal F$ are the classes of torsion groups and finite groups, respectively. In \cite{Dys}, Dyson asked whether the inclusion in (\ref{tf}) is actually an equality. The positive answer would imply some surprising results, e.g., the absence of finitely presented infinite groups of finite exponent, as observed by  Fine, Gaglione, and Spellman \cite{FGS}. Our first goal is to answer this question.

\begin{thm}\label{tor}
We have $\t (\mathcal T) \subsetneqq \t(\mathcal F)$.
\end{thm}

Further we show that the universal theory of some natural classes of torsion groups is undecidable. Recall that a theory $T$ in a language $\mathcal L$ is {\it decidable}, if there is an algorithm which, given a sentence $\Phi $ in $\mathcal L$, decides whether $\Phi \in T$. In \cite{Slo}, Slobodskoi proved that the universal theory of all finite groups is undecidable. In fact, he also proved that the universal theory of all torsion groups is undecidable although it is not stated explicitly in \cite{Slo}. We recover the later result and obtain some new ones.

\begin{thm}\label{und} The universal theory of the following classes of groups is undecidable.
\begin{enumerate}
\item[a)] (Slobodskoi) The class of torsion groups.
\item[b)] The class of groups of any fixed sufficiently large odd non-simple exponent.
\item[c)] The class of $p$-groups for every fixed prime $p\ne 2$.
\end{enumerate}
\end{thm}

Finally we observe that torsion groups can be used to study universal theory of groups of completely different nature, namely hyperbolic and lacunary hyperbolic groups. Recall that a finitely generated group $G$ is {\it lacunary hyperbolic} if at least one asymptotic cone of $G$ is a real tree \cite{OOS}. If, in addition, $G$ is finitely presented, then it is hyperbolic. Thus lacunary hyperbolic groups can be thought of as infinitely presented analogues of hyperbolic groups. Alternatively, one may characterize lacunary hyperbolic groups as directed limits of sequences of hyperbolic groups with some additional properties. This characterization is used in \cite{OOS} to show that lacunary hyperbolic groups resemble hyperbolic ones in many respects.

Recall that every two non-abelian free groups have the same universal theory and this theory is decidable. This solution of the long-standing Tarski Problem was independently obtained by Kharlampovich, Myasnikov \cite{KM} and Sela \cite{S06}. The question of whether the elementary theory of every hyperbolic group is decidable is still open.  A partial result in this direction was obtained by Sela in \cite{S}, where he proved that the universal theory of every torsion-free hyperbolic group is decidable. Sela's result was generalized to all hyperbolic groups (possibly with torsion) by Dahmani and Guirardel \cite{DG}. The primary goal of our last theorem is to show that there is no hope to generalize it to lacunary hyperbolic groups.

\begin{thm}\label{hyp}
Let $\mathcal H$ denote the class of all hyperbolic groups.
\begin{enumerate}
\item $\t (\mathcal H)$ is undecidable.

\item There exists a lacunary hyperbolic group $G$ such that
$\t (G)=\t (\mathcal H)$. In particular, $\t (G)$ is undecidable.
\end{enumerate}
\end{thm}

{\bf Acknowledgments.} I am grateful to Ben Fine, Daniel Groves, Alexander Olshanskii, and Mark Sapir for useful discussions and remarks.

\section{Torsion groups}

Given a subset $S\subseteq G$ of a group $G$, we denote by $\langle S\rangle ^G$ the smallest normal subgroup of $G$ containing $S$. Recall that a subgroup $H\le G$ is a {\it $Q$-subgroup} of $G$ (or has the {\it congruence extension property} in terminology of \cite{OS}) if for every normal subgroup $N$ of $H$, we have $\langle N\rangle ^G\cap H=N$. Equivalently, for every $N\lhd H$, the natural homomorphism from $H/N$ to $G/\langle N\rangle ^G$ is injective. In what follows we denote by $\mathfrak B _n$ the variety of all groups of exponent $n$, i.e., groups satisfying the identity $X^n=1$. By $B(m,n)$ we denote the free Burnside group of rank $m$, i.e., the $m$-generated free group in $\mathfrak B_n$.

We will use the following theorem of Olshanskii and Sapir proved in \cite[Theorem 1.3]{OS}. The proof was later simplified by Ivanov \cite{Iva}.

\begin{thm}[Olsanskii-Sapir, Ivanov]\label{Iva}
For every integer $m>1$ and every odd sufficiently large $n$ (e.g.,$n>2^{48}$), there exists an embedding of $B(m,n)$ into a finitely presented group $G$ with the following properties.
\begin{enumerate}
\item[a)] $B(m,n)$ is a $Q$-subgroup of $G$.

\item[b)] The restriction of the natural homomorphism $\e\colon
G\to G/G^n$ to $B(m,n)$ is injective and $\e(B(m,n))$ is a
$Q$-subgroup of $G/G^n$.
\end{enumerate}
\end{thm}

Theorem \ref{tor} is a particular case of the following more general result.

\begin{thm}
Let $\mathcal C$ be a class of groups such that $\mathfrak B_n\subseteq \mathcal C$ for some odd $n>2^{48}$.
Then $\t (\mathcal C) \subsetneqq \t (\mathcal C\cap \mathcal F)$.
\end{thm}

\begin{proof}
The inclusion $\t (\mathcal C) \subseteq \t (\mathcal C\cap \mathcal F)$ is obvious, so we only need to show that $\t (\mathcal C) \ne \t (\mathcal C\cap \mathcal F)$. Let
\begin{equation}\label{G}
G=\langle x_1, \ldots , x_s \mid R_1, \ldots , R_t\rangle
\end{equation}
be the finitely presented group from Theorem \ref{Iva} that contains $B(2,n)$. Let $w$ be a nontrivial element of $B(2,n)$ whose image in every finite quotient of $B(2,n)$ is trivial. The existence of such an element follows immediately from the solution of the Bounded Burnside Problem by Novikov, Adian (for odd exponents) and Ivanov (in the general case) \cite{Ad,Iva}, and the solution of the Restricted Burnside Problem by Zelmanov \cite{Z1,Z2}. We think of $w$ as an element of $G$ and choose a word $W$ in the alphabet $\{ x_1^{\pm 1}, \ldots , x_s^{\pm 1}\}$ that represents $w$.

Clearly $W$ represents $1$ in every finite quotient of $G$. Thus the universal sentence
$$
L = \forall x_1\ldots \forall x_s\, (R_1=1 \,\&\, \ldots\, \&\, R_t=1 \, \Rightarrow W=1)
$$
is true in every finite group. Indeed if elements $x_1, \ldots, x_s$ of a finite group $K$ satisfy the relations $R_1=1, \ldots , R_t=1$, then the subgroup of $K$ generated by $x_1, \ldots, x_s$ is a finite quotient of $G$.

Thus $L\in \t (\mathcal F)\subseteq \t (\mathcal C \cap \mathcal F)$. On the other hand $L$ is not true in $G/G^n$. Indeed the images of $x_1, \ldots, x_s$ under the natural homomorphism $\e\colon G\to G/G^n$ still satisfy the relations $R_1=1, \ldots , R_t=1$. However $W$ represents the element $\e (w)$ in $G/G^n$, which is nontrivial since the restriction of $\e $ to $B(2,n)$ is injective by Theorem \ref{Iva}. Thus $L\notin \t (\mathfrak B_n)$. As $\mathfrak B_n\subseteq \mathcal C$ we have $\t (\mathcal C)\subseteq \t (\mathfrak B_n)$, and hence $L\notin \t(\mathcal C)$.
\end{proof}

To prove our next theorem we need a particular case of a result of Kharlampovich \cite{Kh}.

\begin{thm}[Kharlampovich]\label{Kh}
Suppose that $n=pq$, where $p\ge 3$ is a prime and $q$ has an odd divisor $\ge 665$. Then for some $m\ge 2$ there exists a normal subgroup $N\le B(m,n)$ such that $N$ is the normal closure of finitely many elements in $B(m,n)$ and the word problem in $B(m,n)/N$ is undecidable.
\end{thm}

Theorem \ref{und} is an immediate corollary of the result below.

\begin{thm}\label{und1}
Suppose that $\mathcal C$ is a class of groups such that $\mathfrak B_n\subseteq \mathcal C$ for some odd non-prime $n>2^{48}$. Then $\t (\mathcal C)$ is undecidable.
\end{thm}

\begin{proof} Clearly every non-prime odd $n>2^{48}$ satisfies the hypothesis of Theorem \ref{Kh}. Let $N$ be a normal subgroup of $B(m,n)$, $m\ge 2$, such that $N$ is the normal closure of finitely many elements in $B(m,n)$ and the word problem in $B(m,n)/N$ is undecidable. Further let $G$ be the finitely presented group from Theorem \ref{Iva} that contains $B(m,n)$.

Let $A$ be the natural image of $B(m,n)$ in  $G_1=G/\langle N\rangle ^G$. Since $B(m,n)$ is a $Q$-subgroup of $G$, we have $A\cong B(m,n)/N$. Note that $G_1$ is finitely presented as $N$ is the normal closure of finitely many elements in $B(m,n)$. Let us fix some finite presentation $$G_1=\langle x_1, \ldots , x_s \mid U_1, \ldots , U_p\rangle .$$

Similarly let $H=G/G^n$ and $H_1=H/\langle \e (N)\rangle ^{H}$, where $\e\colon G\to H$ is the natural homomorphism. Let also $B$ denote the image of $\e(B(m,n))$ in $H_1$. Since the restriction of $\e$ to $B(m,n)$ is injective and $\e (B(m,n))$ is a $Q$-subgroup of $H$ we have $B\cong \e(B(m,n))/\e(N)\cong B(m,n)/N$. Moreover, it is easy to see that there exists a homomorphism $\gamma \colon G_1\to H_1$ which induces an isomorphism $A\to B$.

Let $y_1, \ldots, y_m$ be generators of $A$, $Y_1, \ldots , Y_m$ words in the alphabet $\mathcal X =\{ x_1^{\pm 1}, \ldots , x_s^{\pm 1}\}$ that represent these elements in $G_1$. Given a word $V=V(y_1, \ldots , y_m)$ in the alphabet $\mathcal Y=\{y_1^{\pm 1}, \ldots , y_m^{\pm 1}\}$, let $\widehat V(x_1, \ldots , x_s)$ denote the word obtained from $V$ by replacing $y_i$ with $Y_i$ for every $i=1, \ldots , m$. For every word $V$ in $\mathcal Y$, let
$$
L (V) = \forall x_1\ldots \forall x_s\, (U_1=1\, \&\, \ldots\, \& \, U_p=1 \, \Rightarrow \widehat V(x_1, \ldots , x_s)=1).
$$

Note that if $L(V)\in \t (\mathcal C)$, then $L$ is true in $H_1$ as $H_1\in \mathfrak B_n\subseteq \mathcal C$. Therefore $\widehat V(\gamma (x_1), \ldots , \gamma (x_s))=1$ in $B$. Since $\gamma $ is injective on $A$, $\widehat V(x_1, \ldots , x_s)=1$ in $A$ and hence $V=1$ in $A$. On the other hand, if $L(V)\notin \t (\mathcal C)$, then  there exists a group $C\in \mathcal C$ such that $L(V)$ is false in $C$. This means that there exists a homomorphism $\alpha \colon G_1\to C$ such that $\widehat V(\alpha (x_1), \ldots , \alpha (x_s))\ne 1 $ in $C$. Clearly this implies $V\ne 1$ in $A$.

Thus if we were able to decide if $L(V)\in \t (\mathcal C)$ for every word $V$ in the alphabet $\mathcal Y$, we would be able to solve the word problem in $A\cong B(m,n)/N$. Since the latter problem is undecidable, $\t (\mathcal C)$ is undecidable as well.
\end{proof}

\section{Hyperbolic groups}

The following result is an immediate consequence of methods of \cite{Slo}, although it is not stated explicitly there.

\begin{lem}[Slobodskoi]\label{slo}
Let $T$ be any theory such that
\begin{equation}\label{ttf}
\t (\mathcal T)\subseteq T\subseteq \t (\mathcal F).
\end{equation}
Then $T$ is undecidable.
\end{lem}

\begin{proof}
In \cite{S}, Slobodskoi constructed an effective sequence of formulas $\Psi (k)$, $k\in \mathbb N$, and two (disjoint) recursively inseparable subsets $X,Y\in \mathbb N$ such that $\Psi (k)$ holds for every torsion group whenever $k\in X$ and $\Psi (k)$ does not hold in some finite group whenever $x\in Y$. Recall that two disjoint subsets $X,Y\subseteq \mathbb N$ are said to be {\it recursively inseparable} if there is no recursive subset $I\subseteq \mathbb N$ such that $X\subseteq I$ and $I\cap Y=\emptyset $.

Let now $T$ be a decidable theory that satisfies (\ref{ttf}). Set $I=\{ k\in \mathbb N\mid \Psi (k)\in T\} $. Since $T$ is decidable, $I$ is recursive. Clearly (\ref{ttf}) implies $X\subseteq I$ and $I\cap Y=\emptyset $, which contradicts recursive inseparability of $X$ and $Y$.
\end{proof}

The next lemma is quite elementary. In fact, it follows from Malstsev's Local Theorem (see, e.g., \cite[Theorem 27.3.3]{KM}) applied to groups considered as algebraic systems with one ternary predicate instead of multiplication. We provide a proof for convenience of the reader.
\begin{lem}\label{mal}
Let $G$ be a group and $\mathcal C$ is a class of groups. Suppose that at least one of the following conditions holds.
\begin{enumerate}
\item[(a)] There exists a sequence of normal subgroups $N_1\supseteq N_2\supseteq \ldots $ of $G$ such that $\bigcap\limits_{i=1}^\infty N_i =\{ 1\} $ and $G/N_i\in \mathcal C$ for all $i\in \mathbb N$.
\item[(b)] There is a group $G_0$ and a sequence of normal subgroups $N_1\subseteq N_2\subseteq \ldots $ of $G_0$ such that $G=G_0/\bigcup\limits_{i=1}^\infty N_i$ and $G_0/N_i\in \mathcal C$ for every $i\in \mathbb N$.
    \end{enumerate}
Then $\t (\mathcal C)\subseteq\t (G)$.
\end{lem}

\begin{proof}
 First assume that (a) holds. Let $\mathcal U$ be a non-principal ultrafilter, $P=\prod\limits_{i=1}^\infty (G/N_i)/\mathcal U$ the corresponding ultraproduct of $(G/N_i)$'s. It is straightforward to verify that the map $\e\colon G\to P$ defined by $\e (g)=(gN_1, gN_2, \ldots )$ is a homomorphism and ${\rm Ker }(\e )=\bigcap\limits_{i=1}^\infty N_i =\{ 1\}$. Suppose that a first order formula $\Phi$ is true in $G/N_i$ for all $i\in \mathbb N$. By the \L o\'s Theorem $\Phi $ is true in $P$. If, in addition, $\Phi $ is universal, then it is obviously true in every subgroup of $P$. In particular, $\Phi $ is true in $\e(G)\cong G$.

Suppose now that (b) holds. Let $\mathcal U$ be a non-principal ultrafilter, $\e \colon G_0\to \prod\limits_{i=1}^\infty (G_0/N_i)/\mathcal U$ the map defined by $\e(g)=(gN_1, gN_2, \ldots )$. It is easy to verify that $\e $ is a homomorphism and ${\rm Ker} (\e )= \bigcup\limits_{i=1}^\infty N_i$. Hence $\e (G_0)\cong G$. The rest of the proof is the same as in the first case.
\end{proof}

Recall that a geodesic metric space $M$ is $\delta $-hyperbolic if for every geodesic triangle $\Delta $ in $M$, every side of $\Delta $ belongs to the union of the closed $\delta $-neighborhoods of the other two sides. A group $G$ generated by a set $X$ is $\delta $-hyperbolic relative to $X$ if its Caley graph with respect to $X$ is a $\delta $-hyperbolic metric space. Further let $\e$ be a homomorphism from the group $G$ to another group. The {\it injectivity radius} of $\e$ with respect to $X$, denoted $IR_X(\alpha )$, is defined to be the supremum of all $r>0$ such that $\e$ is injective on a ball of radius $r$ in $G$ (with respect to the word metric corresponding to $X$). In particular, $IR_X(\alpha )=\infty $ if $\e $ is injective. We will make use of the following characterization of lacunary hyperbolic groups proved in \cite[Theorem 1.1]{OOS}.

\begin{thm}[Olshanskii, Osin, Sapir]\label{oos}
A finitely generated group $G$ is lacunary hyperbolic if and only if $G$ is the direct limit of a sequence of finitely generated
groups and epimorphisms
$
G_0\stackrel{\e _1}\longrightarrow
G_1\stackrel{\e _2}\longrightarrow \ldots $ such that $G_i$ is
generated by a finite set $S_i$, $\e _i(S_i)=S_{i+1}$, each
$G_i$ is $\delta _i$--hyperbolic with respect to $S_i$, and
$\delta _i=o(IR_{S_i}(\e _i))$ as $i\to \infty$.
\end{thm}

The construction of the lacunary hyperbolic group in Theorem \ref{hyp} is based on the techniques developed in \cite{SCT}, where the methods suggested by Gromov \cite{Gro} and elaborated by Olshanskii \cite{Ols} for hyperbolic groups are generalized to relatively hyperbolic groups. We heavily rely on results from \cite{AMO,RHG,SCT} and refer to these papers for more details.

Let $K$ be a group hyperbolic relative to a collection of subgroups $\Hl $. A subgroup $L\le K$ is {\it suitable} if it is not virtually cyclic, is not conjugate to a subgroup of one of $H_\lambda $'s, and does not normalize any nontrivial finite normal subgroup of $K$. This definition is different from the original one from \cite{SCT}, but is equivalent to the latter by \cite[Proposition 3.4]{AMO}. Parts (a)-(d) of the following theorem are proved in \cite[Theorem 2.4]{SCT}. The last part follows immediately from \cite[Lemma 5.1]{SCT} and the proof of \cite[Theorem 2.4]{SCT}.

\begin{thm}\label{glue}
Let $K$ be a group hyperbolic relative to a collection of
subgroups $\Hl $, $L$ a suitable subgroup of $K$, $t_1, \ldots
, t_m$ arbitrary elements of $K$, and $S\subseteq K$ a finite subset. Then there exists an epimorphism
$\eta \colon K\to \overline{K}$ such that:
\begin{enumerate}
\item[(a)] The restriction of $\eta $ to $\bigcup\limits_{\lambda\in \Lambda
}H_\lambda $ is injective.

\item[(b)] The group $\overline{K}$ is hyperbolic relative to $\{ \eta
(H_\lambda ) \} _{\lambda \in \Lambda } $.

\item[(c)] For any $i=1, \ldots , m$, we have $\eta (t_i)\in \eta (H)$.

\item[(d)] $\eta (L)$ is a suitable subgroup of $\overline{K}$.

\item[(e)] The restriction of $\eta $ to $S$ is injective.
\end{enumerate}
\end{thm}

\begin{prop}\label{lh}
There exists a lacunary hyperbolic group $G$ such that every hyperbolic group embeds in $G$.
\end{prop}

\begin{proof}
Since every hyperbolic group is finitely presented, we can enumerate all hyperbolic groups: $H_0=\{ 1\},  H_1, H_2, \ldots \,$ We set $G_0=F(x,y)$, where $F(x,y)$ is the free group freely generated by $x$ and $y$, and construct a sequence of groups and epimorphisms
\begin{equation}\label{seq}
G_0\stackrel{\e_1}\to G_1\stackrel{\e_2}\to
\end{equation}
by induction. Suppose that $G_n$ is already constructed. We use the same notation for elements $x,y$ their images in $G_n$. Assume that the groups $H_0, \ldots , H_n$ properly embed in $G_n$, so we may think of them as subgroups of $G_n$.  Suppose also that $G_n$ is not virtually cyclic, contains no nontrivial finite normal subgroups, and is hyperbolic relative to the collection $\{ H_0, \ldots , H_n\} $. In particular, $G_n$ is hyperbolic by \cite[Corollary 2.41]{RHG}. These conditions obviously hold for $G_0$.

Let $\delta _n$ be the hyperbolicity constant of $G_n$ with respect to the generating set $\{ x, y\} $, $S_n$ the subset of all elements of $G_n$ of length at most $2^{\delta _n}$ with respect to $\{ x, y\} $. Let also $K=G_n\ast H_{n+1}$. It follows from the definition of relative hyperbolicity via isoperimetric functions (see \cite{RHG}) that $K$ is hyperbolic relative to $\{ H_0, \ldots , H_{n+1}\}$. Standard arguments show that $L=G_n$ is a suitable subgroup of $K$. Let $t_1, \ldots , t_k$ be generators of $H_{n+1}$.

We denote by $G_{n+1}$ the quotient group of $K$ that satisfies the conclusion of Theorem \ref{glue}. By part (a) of Theorem \ref{glue} we may think of $H_0, \ldots ,H_{n+1}$ as subgroups of $G_{n+1}$. Further by part (b) $G_{n+1}$ is hyperbolic relative to $\{H_0, \ldots, H_{n+1}\} $. By part (c) the images of $t_i$'s in $G_{n+1}$ belong to the image of $G_n$. Thus $G_{n+1}$ is generated by $\{ x,y\} $ and there is a natural epimorphism $\e_{n+1}\colon G_n\to G_{n+1}$. Note that by part (e) the natural homomorphism $K\to G_{n+1} $ is injective on $S_n$. This implies that
\begin{equation}\label{ir}
IR_{\{ x,y\} }(\e_{n+1}) \ge 2^{\delta _n}
\end{equation}
Finally by part (d) $G_{n+1}$ is a suitable subgroup of itself. In particular, it is non-elementary and contains no nontrivial finite normal subgroups. This completes the inductive step.

Let now $G$ be the direct limit of (\ref{seq}). That is, let $N_i={\rm Ker } (\e _n\circ \cdots \circ \e _1)$, $i=1,2, \ldots $, and $G=G_0/\bigcup\limits_{i=1}^\infty N_i$. Then $G$ is lacunary hyperbolic by (\ref{ir}) and Theorem \ref{oos}. It is clear from our construction that every hyperbolic group embeds in $G$.
\end{proof}

\begin{proof}[Proof of Theorem \ref{hyp}]
1) It easily follows from results of \cite{Ols} that every hyperbolic group is residually torsion. In fact, every hyperbolic group $G$ is even residually torsion of bounded exponent, that is, $\bigcap_{n=1}^\infty G^n =\{ 1\} $, where $G^n$ is the subgroup generated by $\{ g^n\mid g\in G\}$. The later result is proved in \cite{IO}.

Let $N_i=\bigcap_{n=1}^i G^n $. Then the sequence $\{ N_i\}$ satisfies the first condition from Lemma \ref{mal} for $\mathcal C=\mathcal T$. Hence $\t (\mathcal T)\subseteq \t (G)$ for every hyperbolic group $G$. Consequently $\t (\mathcal T)\subseteq \t(\mathcal H)$. On the other hand we have $\t(\mathcal H)\subseteq \t (\mathcal F)$ as every finite group is hyperbolic. Therefore, $\t (\mathcal H)$ is undecidable by Lemma \ref{slo}.

2) Let $G$ be the group from Proposition \ref{lh}. Since every hyperbolic group embeds in $G$, we have $\t (G)\subseteq \t (\mathcal H)$. On the other hand, $\t (\mathcal H)\subseteq \t (G)$ for every lacunary hyperbolic group $G$ by Lemma \ref{mal} and Theorem \ref{oos}. Thus $\t (G)=\t (\mathcal H)$.
\end{proof}

\vspace{1cm}

\noindent \small  Stevenson Center 1326, Department of
Mathematics, Vanderbilt University\\  Nashville, TN 37240, USA

\vspace{3mm}

\noindent \small \it E-mail address: \tt
denis.osin@gmail.com

\end{document}